\documentclass[10pt,twoside]{amsart}
\usepackage{amsmath,amssymb,fullpage}
\usepackage{graphicx,psfrag}

\def\cvd{~\vbox{\hrule\hbox{%
     \vrule height1.3ex\hskip0.8ex\vrule}\hrule } }

\usepackage[english]{babel}
\usepackage{graphicx,epstopdf,epsfig}
\usepackage{amsfonts,epsfig,fancyhdr,graphics, hyperref,amsmath,amssymb}
\usepackage{epstopdf}
\usepackage{subfigure}
\usepackage{color}
\usepackage{xcolor}
\usepackage{colortbl}
\usepackage{enumerate}
\usepackage{tikz}
\usepackage{amsmath}
\usepackage{pgf, tikz}
\usetikzlibrary{arrows, automata}
\usepackage{graphicx}
\usepackage{amssymb}
\usepackage{mathrsfs}
\usepackage[all]{xy}
\usepackage{amsfonts}
\usepackage{float}
\usepackage{enumerate}
\usepackage{amsmath}
\usepackage{cases}
\usepackage{amsfonts}
\usepackage{graphicx,psfrag}
\usepackage{array}
\usepackage{colortbl}
\usepackage{amsmath}
\usepackage{fancyhdr}
\usepackage{epstopdf}
\usepackage{subfigure}
\usepackage{color}

\newtheorem{theorem}{Theorem}[section]
\newtheorem{corollary}[theorem]{Corollary}
\newtheorem{lemma}[theorem]{Lemma}
\newtheorem{conjecture}[theorem]{Conjecture}

\newtheorem{example}[theorem]{Example}

\newcommand{\Names}{Megan Bennett and Lei Cao}
\newcommand{\Title}{Flag-Shaped Blockers of 123-Avoiding Permutation Matrices}

\title[Flag-Shaped Blockers of 123-Avoiding Permutation Matrices]{Flag-Shaped Blockers of 123-Avoiding Permutation Matrices}

\author{Megan Bennett$^\ast$
\and
Lei Cao$^{\ast\ast}$}

\thanks{ $^\ast$  Corresponding Author, Department of Mathematics, Halmos College of Arts \& Sciences and Farquhar Honors College, Nova Southeastern University, Fort Lauderdale, FL 33328, USA (mb3836@mynsu.nova.edu). Supported by the NSU PanSGA grant, Farquhar Honors College Travel Funding, and NSU pfrdg-334925.}

\thanks{ $^{\ast\ast}$ Department of Mathematics, Halmos College of Arts \& Sciences, Nova Southeastern University, Fort Lauderdale, FL 33328, USA (lcao@nova.edu). Supported by NSU pfrdg-334925}

\begin{document}

\bibliographystyle{plain}

\setcounter{page}{1}

\thispagestyle{empty}

\maketitle

\begin{abstract}
A blocker of $123$-avoiding permutation matrices refers to the set of zeros contained within an $n\times n$ $123$-forcing matrix. Recently, Brualdi and Cao provided a characterization of all minimal blockers, which are blockers with a cardinality of $n$. Building upon their work, a new type of blocker, flag-shaped blockers, which can be seen as a generalization of the $L$-shaped blockers defined by Brualdi and Cao, are introduced. It is demonstrated that all flag-shaped blockers are minimum blockers. The possible cardinalities of flag-shaped blockers are also determined, and the dimensions of subpolytopes that are defined by flag-shaped blockers are examined.
\end{abstract}

\textbf{Key words.}
123-pattern, 123-avoiding permutation matrices, blockers

\textbf{AMS subject classifications.}
05A05, 15A03, 15B99, 52A20. 

\section{Introduction}
Let $n$ be a positive integer and let ${\mathcal P}_n$ be the set of $n\times n$ permutation matrices corresponding to the set ${\mathcal S}_n$ of permutations of $\{1,2,\ldots,n\}$.
A permutation $\sigma$ of $\{1,2,\ldots,n\}$ {\it contains a $123$-pattern} provided it contains an increasing subsequence of length $3$ and, otherwise, is {\it $123$-avoiding}. In terms of the $n\times n$ permutation matrix $P$ corresponding to $\sigma$, $P$ contains a $123$-pattern provided the $3\times 3$ identity matrix  $I_3$ is a submatrix of $P$.

If $A$ is an $n\times n$ $(0,1)$-matrix, then $A$ is {\it $123$-forcing} provided every permutation matrix $P\le A$ (entrywise order) contains a $123$-pattern; the matrix $A$ thus {\it blocks} all $123$-avoiding permutations in that every $123$-avoiding permutation matrix has at least one $1$ in a position of a $0$ of $A$. The set of zeros contained in a $123$-forcing matrix $A$ is called a blocker of $123$-avoiding permutation matrices. A blocker $\mathcal{B}$ is \textit{minimum} if removing any element from $\mathcal{B}$ makes it no longer a blocker. If there does not exist a blocker $\mathcal{D}$ such that the cardinality of $\mathcal{D}$ is less than the cardinality of $\mathcal{B},$ then $\mathcal{B}$ is \textit{minimal}. Thus, all minimal blockers are minimum. $123$-forcing matrices and blockers of $123$-avoiding permutation matrices have been previously investigated in \cite{BC,BC2,BC1}.

Denote the $n\times n$ matrix of all $1'$s as $J_n$. We define the cyclic-Hankel decomposition of $J_n$ into $n$ permutation matrices by starting with row $1$ and cyclically permuting it as for circulant matrices but in a right-to-left fashion, rather than the left-to-right fashion. We call these permutation matrices cyclic-Hankel permutation matrices or cyclic-Hankel diagonals. The cyclic-Hankel decomposition is illustrated for $n=6$ using letters $a, b, c, d, e, f$ below:
\[\left[\begin{array}{c|c|c|c|c|c}
a&b&c&d&e&f\\ \hline
b&c&d&e&f&a\\ \hline
c&d&e&f&a&b\\ \hline
d&e&f&a&b&c\\ \hline
e&f&a&b&c&d\\ \hline
f&a&b&c&d&e\end{array}\right].\]

\begin{lemma} [Lemma 2.2 in \cite{BC}] The number of $0$'s in an $n\times n$ $123$-forcing $(0,1)$-matrix is at least $n$. An $n\times n$ $123$-forcing $(0,1)$-matrix with exactly n $0$'s contains exactly one $0$ from each cyclic-Hankel permutation matrix.
\end{lemma}

One particular importance of all blockers with exact $n$ zeros is the $L$-shaped blockers, a set of $n$ adjacent positions denoted $L_n(s,r)$, where $s$ corresponds to the width of the blocker, $r$ corresponds to the height, and $r+s=n+1$. Clearly, an $n\times n$ $(0,1)$-matrix $A$ with a row or column of all $0$'s is $123$-forcing, since there do not exist any permutation matrices $P\le A$. If the zeros are not contained in one row or one column,  $L$-shaped blockers must contain either the $(1,n)$ or $(n,1)$ position. In \cite{BC}, it was shown that all blockers with exact $n$ zeros can be obtained from  $L$-shaped blockers by shifting some zeros along the cyclic-Hankel diagonals. The following examples illustrate these results.

 \begin{example}{\rm \label{ex:123avoid} For $n=6$, one possible $L$-shaped blocker is
 \[L_6(4,3)=\left[\begin{array}{c|c|c|c|c|c}
 1&1&0&0&0&0\\ \hline
  1&1&1&1&1&0\\ \hline
  1 &1&1&1&1&0\\ \hline
    1&1&1&1&1&1\\ \hline
    1 &1&1&1&1&1\\ \hline
     1 &1&1&1&1&1\end{array}\right].\]

Every permutation matrix $P\le L_6(4,3)$ contains one of the two $1$'s from row $1$, one of the three $1$'s from column $6$, and then necessarily one of the $1$'s from the $2\times 3$ submatrix formed by rows $2$ and $3$, and columns $3, 4$, and $5$, as given by the Frobenius-K\"{o}nig theorem \cite{FK1933}, thereby resulting in a $123$-pattern. Thus $A$ is a $123$-forcing matrix; equivalently, $A$ blocks all $6\times 6$ $123$-avoiding permutation matrices. Another example of a ${123}$-forcing matrix with 6\ $0$'s obtained from $L_6(4,3)$ by shifting some zeros along the cyclic-Hankel diagonals is
\[\left[\begin{array}{c|c|c|c|c|c}
1&1&1&0&0&0\\ \hline
1&0&1&1&1&0 \\ \hline
1&1&1&1&1&1 \\ \hline
1&1&1&1&0&1 \\ \hline
1&1&1&1&1&1 \\ \hline
1&1&1&1&1&1 \end{array}\right].\]
 }\hfill{$\cvd$}
 \end{example}

 The main purpose of this paper is to characterize $123$-forcing matrices (equivalently, blockers of $123$-avoiding permutation matrices) $A$ with $0$'s in certain shapes. We now briefly summarize the content of this paper. In Section 2, we define flag-shaped blockers of $123$-avoiding permutation matrices and show that they are minimum blockers. In Section 3, we give all possible cardinalities (number of zeros) of an $n\times n$ flag-shaped blocker of $123$-avoiding permutation matrices and propose a conjecture that any $n\times n$ minimum blocker of $123$-avoiding permutation matrices contains at least $n$ zeros and at most $rs$ zeros, where $|r-s|\leq 1$ and $r+s=n+1.$ In Section 4, we explore the dimensions of the subpolytopes of a $123$-avoiding polytope determined by a flag-shaped blocker.

\section{Defining Flag-Shaped Blockers}\label{Section 2}

A flag-shaped blocker of an $n \times n$ matrix is a set of adjacent positions, denoted $B_n(m,t)$, in the following form.

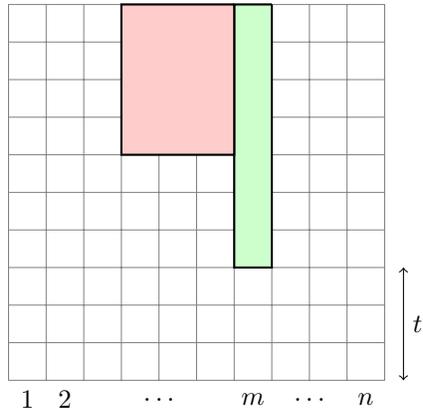
\begin{center}
\begin{tikzpicture}
    \draw[step=5mm,gray,very thin] (0,0) grid (5,5);
    \filldraw[fill=green!20!white, draw=black, thick] (3.5,5) -- (3,5) -- (3,1.5) -- (3.5,1.5) -- (3.5,5);
    \filldraw[fill=red!20!white, draw=black, thick] (3,5) -- (1.5,5) -- (1.5,3) -- (3,3) -- (3,5);
    \node at (0.25,-0.25) {$1$};
    \node at (0.75,-0.25) {$2$};
    \node at (2,-0.25) {$\hdots$};
    \node at (3.25,-0.25) {$m$};
    \node at (4,-0.25) {$\hdots$};
    \node at (4.75,-0.25) {$n$};
    \draw[<->] (5.25,0) -- (5.25, 1.5) node[scale=1, midway, right]{$t$};
\end{tikzpicture}
\end{center}

The blocker is composed of two parts, the "pole," shown in green, and the "flag," shown in red. The "pole" is a $(n-t)\times 1$ submatrix contained in the $m^{th}$ column, where $1\leq m\leq n.$ The "flag" is an $(n-m+1)\times t$ submatrix with consecutive rows and columns contained in row $1$ to row $n-m+1$ and column $m-t$ to column $m-1,$  where $t$ is the number of unoccupied rows below the blocker. We require the height of the pole to be greater or equal to the height of the flag, hence $0 \leq t \leq m-1$. Note that $B_n(m,t)$ has cardinality $n+t(n-m).$

Since the $123$-avoiding property is preserved when we take the transpose, take the Hankel transpose, or rotate the matrix by $180^{\circ},$ a flag-shaped blocker is still a blocker if it is reflected across the main or Hankel diagonal or rotated $180^{\circ}$ about the center of the matrix. Thus, without loss of generality, we will only consider the case where flag-shaped blockers are oriented in the same manner as the figure above, unless stated otherwise.

Interestingly, these flag-shaped blockers can be considered generalizations of $L$-shaped blockers, which are introduced in \cite{BC2}, as an $L$-shaped blocker with $m=n$ is a flag-shaped blocker. Furthermore, the flag-shaped blockers $B_n(m,m-1)$ are special types of blockers given by the Frobenius-K\"{o}nig theorem. Thus, $L$-shaped blockers and certain cases of Frobenius-K\"{o}nig blockers are both special cases of flag-shaped blockers.

Starting with Lemma \ref{Lemma2.1}, we first prove that $B_n(m,t),$ the set of positions arranged in a flag-shape, are blockers, then we show they are minimum blockers in Lemma \ref{minimumR1} and Theorem \ref{Theorem2.3}. In the following proofs and examples, we will utilize the cyclic-Hankel decomposition with red positions denoting blocker positions, unless specified otherwise.

\begin{lemma}\label{Lemma2.1} Let $Q$ be a set of positions of an $n \times n$ matrix. If $Q$ is a flag-shaped blocker $B_n(m,t)$ with cardinality $n+t(n-m)$, where $0 \leq t \leq m-1$ and $1 \leq m \leq n$, then $Q$ is a blocker of all $n \times n$ $123$-avoiding permutation matrices.
\end{lemma}

\begin{proof}
Without loss of generality, we consider a $10 \times 10$ matrix with the flag-shaped blocker $B_n(7,3)$ for illustrative purposes, though the general proof follows in the same manner.

We know that if there exists a $123$-avoiding permutation matrix that does not intersect $Q$, then it must contain one of the positions in yellow submatrix in the below matrix.

\[\left[\begin{array}{c|c|c|c|c|c|c|c|c|c}
a&b&c&\cellcolor{red!20}d&\cellcolor{red!20}e&\cellcolor{red!20}f&\cellcolor{red!20}g&h&i&j\\ \hline
b&c&d&\cellcolor{red!20}e&\cellcolor{red!20}f&\cellcolor{red!20}g&\cellcolor{red!20}h&i&j&a\\ \hline
c&d&e&\cellcolor{red!20}f&\cellcolor{red!20}g&\cellcolor{red!20}h&\cellcolor{red!20}i&j&a&b\\ \hline
d&e&f&\cellcolor{red!20}g&\cellcolor{red!20}h&\cellcolor{red!20}i&\cellcolor{red!20}j&a&b&c\\ \hline
e&f&g&\cellcolor{green!20}h&\cellcolor{green!20}i&\cellcolor{green!20}j&\cellcolor{red!20}a&b&c&d\\ \hline
f&g&h&\cellcolor{green!20}i&\cellcolor{green!20}j&\cellcolor{green!20}a&\cellcolor{red!20}b&c&d&e\\ \hline
g&h&i&\cellcolor{green!20}j&\cellcolor{green!20}a&\cellcolor{green!20}b&\cellcolor{red!20}c&d&e&f\\ \hline
h&i&j&a&b&c&\cellcolor{yellow!20}d&e&f&g\\ \hline
i&j&a&b&c&d&\cellcolor{yellow!20}e&f&g&h\\ \hline
j&a&b&c&d&e&\cellcolor{yellow!20}f&g&h&i
\end{array}\right]\]

We also know that at least one of the green positions in the above figure must be used to construct the permutation matrix, as the set of green and red positions forms a blocker given by the Frobenius-K\"{o}nig theorem, which intersects every permutation matrix.

The yellow submatrix in the upper right corner shown in the figure below will always contain more rows, $n-m+1$, than columns, $n-m$. Thus, it is necessary that we use at least one position from the green submatrix in the upper left corner in the figure below to construct the permutation matrix.

\[\left[\begin{array}{c|c|c|c|c|c|c|c|c|c}
\cellcolor{green!20}a&\cellcolor{green!20}b&\cellcolor{green!20}c&\cellcolor{red!20}d&\cellcolor{red!20}e&\cellcolor{red!20}f&\cellcolor{red!20}g&\cellcolor{yellow!20}h&\cellcolor{yellow!20}i&\cellcolor{yellow!20}j\\ \hline
\cellcolor{green!20}b&\cellcolor{green!20}c&\cellcolor{green!20}d&\cellcolor{red!20}e&\cellcolor{red!20}f&\cellcolor{red!20}g&\cellcolor{red!20}h&\cellcolor{yellow!20}i&\cellcolor{yellow!20}j&\cellcolor{yellow!20}a\\ \hline
\cellcolor{green!20}c&\cellcolor{green!20}d&\cellcolor{green!20}e&\cellcolor{red!20}f&\cellcolor{red!20}g&\cellcolor{red!20}h&\cellcolor{red!20}i&\cellcolor{yellow!20}j&\cellcolor{yellow!20}a&\cellcolor{yellow!20}b\\ \hline
\cellcolor{green!20}d&\cellcolor{green!20}e&\cellcolor{green!20}f&\cellcolor{red!20}g&\cellcolor{red!20}h&\cellcolor{red!20}i&\cellcolor{red!20}j&\cellcolor{yellow!20}a&\cellcolor{yellow!20}b&\cellcolor{yellow!20}c\\ \hline
e&f&g&h&i&j&\cellcolor{red!20}a&b&c&d\\ \hline
f&g&h&i&j&a&\cellcolor{red!20}b&c&d&e\\ \hline
g&h&i&j&a&b&\cellcolor{red!20}c&d&e&f\\ \hline
h&i&j&a&b&c&d&e&f&g\\ \hline
i&j&a&b&c&d&e&f&g&h\\ \hline
j&a&b&c&d&e&f&g&h&i
\end{array}\right]\]

However, this means the permutation matrix constructed cannot be $123$-avoiding without intersecting the blocker, as at least one element from each of the yellow submatrices below must be utilized when constructing a permutation matrix that does not intersect $Q$. Thus, $Q$ is a blocker.

\[\left[\begin{array}{c|c|c|c|c|c|c|c|c|c}
\cellcolor{yellow!20}a&\cellcolor{yellow!20}b&\cellcolor{yellow!20}c&\cellcolor{red!20}d&\cellcolor{red!20}e&\cellcolor{red!20}f&\cellcolor{red!20}g&h&i&j\\ \hline
\cellcolor{yellow!20}b&\cellcolor{yellow!20}c&\cellcolor{yellow!20}d&\cellcolor{red!20}e&\cellcolor{red!20}f&\cellcolor{red!20}g&\cellcolor{red!20}h&i&j&a\\ \hline
\cellcolor{yellow!20}c&\cellcolor{yellow!20}d&\cellcolor{yellow!20}e&\cellcolor{red!20}f&\cellcolor{red!20}g&\cellcolor{red!20}h&\cellcolor{red!20}i&j&a&b\\ \hline
\cellcolor{yellow!20}d&\cellcolor{yellow!20}e&\cellcolor{yellow!20}f&\cellcolor{red!20}g&\cellcolor{red!20}h&\cellcolor{red!20}i&\cellcolor{red!20}j&a&b&c\\ \hline
e&f&g&\cellcolor{yellow!20}h&\cellcolor{yellow!20}i&\cellcolor{yellow!20}j&\cellcolor{red!20}a&b&c&d\\ \hline
f&g&h&\cellcolor{yellow!20}i&\cellcolor{yellow!20}j&\cellcolor{yellow!20}a&\cellcolor{red!20}b&c&d&e\\ \hline
g&h&i&\cellcolor{yellow!20}j&\cellcolor{yellow!20}a&\cellcolor{yellow!20}b&\cellcolor{red!20}c&d&e&f\\ \hline
h&i&j&a&b&c&\cellcolor{yellow!20}d&e&f&g\\ \hline
i&j&a&b&c&d&\cellcolor{yellow!20}e&f&g&h\\ \hline
j&a&b&c&d&e&\cellcolor{yellow!20}f&g&h&i
\end{array}\right]\]
\end{proof}

We now show that flag-shaped blockers $B_n(m,t)$ are minimum blockers.

\begin{lemma} \label{minimumR1} For any element $b$ in the first row of a flag-shaped blocker $B_n(m,t)$ of all $n \times n$ $123$-avoiding permutation matrices, there exists at least one $123$-avoiding permutation matrix that intersects the blocker at most once at $b$.
\end{lemma}

\begin{proof} Without loss of generality, we consider a $10 \times 10$ matrix with the flag-shaped blocker $B_n(7,3)$ for illustrative purposes, though the general proof follows in the same manner.

Recall that if a blocker is minimum, then removing any position will no longer make it a blocker. To prove that all the blocker positions in the first row are necessary, we need to show that a $123$-avoiding permutation matrix can be constructed using any of the positions of the blocker in the first row.


We start by considering the first $m-t-1$ columns of the matrix, the first three columns in the matrix below. We can begin constructing the $123$-avoiding permutation matrix that only intersects the blocker in row 1 by utilizing the highest positions, going from northeast to southwest and starting at row 2. These are the three green positions in the figure below.

\[\left[\begin{array}{c|c|c||c|c|c|c|c|c|c}
a&b&c&\cellcolor{red!20}d&\cellcolor{red!20}e&\cellcolor{red!20}f&\cellcolor{red!20}g&h&i&j\\ \hline
b&c&\cellcolor{green!20}d&\cellcolor{red!20}e&\cellcolor{red!20}f&\cellcolor{red!20}g&\cellcolor{red!20}h&i&j&a\\ \hline
c&\cellcolor{green!20}d&e&\cellcolor{red!20}f&\cellcolor{red!20}g&\cellcolor{red!20}h&\cellcolor{red!20}i&j&a&b\\ \hline
\cellcolor{green!20}d&e&f&\cellcolor{red!20}g&\cellcolor{red!20}h&\cellcolor{red!20}i&\cellcolor{red!20}j&a&b&c\\ \hline
e&f&g&h&i&j&\cellcolor{red!20}a&b&c&d\\ \hline
f&g&h&i&j&a&\cellcolor{red!20}b&c&d&e\\ \hline
g&h&i&j&a&b&\cellcolor{red!20}c&d&e&f\\ \hline
h&i&j&a&b&c&d&e&f&g\\ \hline
i&j&a&b&c&d&e&f&g&h\\ \hline
j&a&b&c&d&e&f&g&h&i
\end{array}\right]\]

Next, we will consider the last $n-m$ columns. We can use positions from these columns to construct the $123$-avoiding permutation matrix, starting with the most northeastern position in the row below the lowest position used in the previous step.

\[\left[\begin{array}{c|c|c|c|c|c|c||c|c|c}
a&b&c&\cellcolor{red!20}d&\cellcolor{red!20}e&\cellcolor{red!20}f&\cellcolor{red!20}g&h&i&j\\ \hline
b&c&\cellcolor{green!20}d&\cellcolor{red!20}e&\cellcolor{red!20}f&\cellcolor{red!20}g&\cellcolor{red!20}h&i&j&a\\ \hline
c&\cellcolor{green!20}d&e&\cellcolor{red!20}f&\cellcolor{red!20}g&\cellcolor{red!20}h&\cellcolor{red!20}i&j&a&b\\ \hline
\cellcolor{green!20}d&e&f&\cellcolor{red!20}g&\cellcolor{red!20}h&\cellcolor{red!20}i&\cellcolor{red!20}j&a&b&c\\ \hline
e&f&g&h&i&j&\cellcolor{red!20}a&b&c&\cellcolor{green!20}d\\ \hline
f&g&h&i&j&a&\cellcolor{red!20}b&c&\cellcolor{green!20}d&e\\ \hline
g&h&i&j&a&b&\cellcolor{red!20}c&\cellcolor{green!20}d&e&f\\ \hline
h&i&j&a&b&c&d&e&f&g\\ \hline
i&j&a&b&c&d&e&f&g&h\\ \hline
j&a&b&c&d&e&f&g&h&i
\end{array}\right]\]

Exactly $n-t-1$ out of $n$ columns have been accounted for, and elements in these columns that we use to construct the $123$-avoiding permutation matrix will never intersect the blocker since there are no positions of the blocker in these columns. This implies that there are exactly $n-t-1$ rows in the matrix that pose no issues when constructing the permutation matrix.

Now, all that is left to consider is row 1 and the last $t$ rows of the matrix, for a total of $t+1$ rows. By definition, the last $t$ rows contain no positions of the blocker, so it is always possible to use $t$ positions from northeast to southwest for the portion of the $123$-avoiding permutation matrix in these last $t$ rows.

So far, the $123$-avoiding permutation matrix being constructed is disjoint from the blocker. However, the last element of the permutation matrix must come from row 1. This element will intersect the blocker, so there exists a permutation matrix that intersects the blocker at most once, as shown in the following examples, where the yellow positions represent the intersections of the blocker and the $123$-avoiding permutation matrix.

\[\left[\begin{array}{c|c|c|c|c|c|c|c|c|c}
a&b&c&\cellcolor{red!20}d&\cellcolor{red!20}e&\cellcolor{red!20}f&\cellcolor{yellow!20}g&h&i&j\\ \hline
b&c&\cellcolor{green!20}d&\cellcolor{red!20}e&\cellcolor{red!20}f&\cellcolor{red!20}g&\cellcolor{red!20}h&i&j&a\\ \hline
c&\cellcolor{green!20}d&e&\cellcolor{red!20}f&\cellcolor{red!20}g&\cellcolor{red!20}h&\cellcolor{red!20}i&j&a&b\\ \hline
\cellcolor{green!20}d&e&f&\cellcolor{red!20}g&\cellcolor{red!20}h&\cellcolor{red!20}i&\cellcolor{red!20}j&a&b&c\\ \hline
e&f&g&h&i&j&\cellcolor{red!20}a&b&c&\cellcolor{green!20}d\\ \hline
f&g&h&i&j&a&\cellcolor{red!20}b&c&\cellcolor{green!20}d&e\\ \hline
g&h&i&j&a&b&\cellcolor{red!20}c&\cellcolor{green!20}d&e&f\\ \hline
h&i&j&a&b&\cellcolor{green!20}c&d&e&f&g\\ \hline
i&j&a&b&\cellcolor{green!20}c&d&e&f&g&h\\ \hline
j&a&b&\cellcolor{green!20}c&d&e&f&g&h&i
\end{array}\right] \ \ or \ \
\left[\begin{array}{c|c|c|c|c|c|c|c|c|c}
a&b&c&\cellcolor{red!20}d&\cellcolor{yellow!20}e&\cellcolor{red!20}f&\cellcolor{red!20}g&h&i&j\\ \hline
b&c&\cellcolor{green!20}d&\cellcolor{red!20}e&\cellcolor{red!20}f&\cellcolor{red!20}g&\cellcolor{red!20}h&i&j&a\\ \hline
c&\cellcolor{green!20}d&e&\cellcolor{red!20}f&\cellcolor{red!20}g&\cellcolor{red!20}h&\cellcolor{red!20}i&j&a&b\\ \hline
\cellcolor{green!20}d&e&f&\cellcolor{red!20}g&\cellcolor{red!20}h&\cellcolor{red!20}i&\cellcolor{red!20}j&a&b&c\\ \hline
e&f&g&h&i&j&\cellcolor{red!20}a&b&c&\cellcolor{green!20}d\\ \hline
f&g&h&i&j&a&\cellcolor{red!20}b&c&\cellcolor{green!20}d&e\\ \hline
g&h&i&j&a&b&\cellcolor{red!20}c&\cellcolor{green!20}d&e&f\\ \hline
h&i&j&a&b&c&\cellcolor{green!20}d&e&f&g\\ \hline
i&j&a&b&c&\cellcolor{green!20}d&e&f&g&h\\ \hline
j&a&b&\cellcolor{green!20}c&d&e&f&g&h&i
\end{array}\right]\]

Thus, for any element $b$ in the flag-shaped blocker that resides in the first row of an $n \times n$ matrix, there exists at least one $123$-avoiding permutation matrix that intersects the blocker at most once at $b$. By definition, this means that every element of the blocker in the first row of the matrix is a necessary position of the blocker. \end{proof}

\begin{theorem}\label{Theorem2.3}
Let $Q$ be a set of positions of an $n \times n$ matrix. If $Q$ is a flag-shaped blocker $B_n(m,t)$ with cardinality $n+t(n-m)$, where $0 \leq t \leq m-1$ and $1 \leq m \leq n$, then $Q$ is a minimum blocker of all $n \times n$ $123$-avoiding permutation matrices.
\end{theorem}

\begin{proof}
Without loss of generality, we consider a $10 \times 10$ matrix with the flag-shaped blocker $B_n(7,3)$ for illustrative purposes, though the general proof follows in the same manner.

By Lemma~\ref{minimumR1}, every element of $Q$ in the first row of the matrix is necessary for the minimum blocker, so we can consider the following matrix.

\[\left[\begin{array}{c|c|c|c|c|c|c|c|c||c}
a&b&c&\cellcolor{red!20}d&\cellcolor{red!20}e&\cellcolor{red!20}f&\cellcolor{red!20}g&h&i&\cellcolor{green!20}j\\ \hline \hline
b&c&d&\cellcolor{red!20}e&\cellcolor{red!20}f&\cellcolor{red!20}g&\cellcolor{red!20}h&i&j&a\\ \hline
c&d&e&\cellcolor{red!20}f&\cellcolor{red!20}g&\cellcolor{red!20}h&\cellcolor{red!20}i&j&a&b\\ \hline
d&e&f&\cellcolor{red!20}g&\cellcolor{red!20}h&\cellcolor{red!20}i&\cellcolor{red!20}j&a&b&c\\ \hline
e&f&g&h&i&j&\cellcolor{red!20}a&b&c&d\\ \hline
f&g&h&i&j&a&\cellcolor{red!20}b&c&d&e\\ \hline
g&h&i&j&a&b&\cellcolor{red!20}c&d&e&f\\ \hline
h&i&j&a&b&c&d&e&f&g\\ \hline
i&j&a&b&c&d&e&f&g&h\\ \hline
j&a&b&c&d&e&f&g&h&i
\end{array}\right]\]

We can use the green $j$ to attempt to construct a $123$-avoiding permutation matrix, allowing us to focus on the $9 \times 9$ submatrix at the lower left corner. In this submatrix, we have another flag-shaped blocker. By repeating the above steps, we can show that all the blocker elements in row 2 must be included in the minimum blocker. This process of reducing the size of the flag-shaped blocker can be repeated until we are left with an $L$-shaped minimal blocker, all positions of which are necessary. Thus, all positions of $Q$ are necessary. Since removing any element of $Q$ makes it no longer a blocker, $Q$ is a minimum blocker.
\end{proof}

\section{The Cardinality of Flag-Shaped Blockers}

Flag-shaped blockers have a cardinality of $n$ only when $m=n$ or $m=1$. More generally, the cardinality of a flag-shaped minimum blocker is given by the expression $n+t(n-m)$, where $0 \leq t \leq m-1$ and $1 \leq m \leq n$. With this expression, we are able to determine which cardinalities can and cannot be achieved by a flag-shaped blocker.

\begin{theorem}\label{Theorem 4}
    Let $n,m,$ and $t$ be nonnegative integers, such that $0 \leq t \leq m-1$ and $1 \leq m \leq n$, and let $p:=n+t(n-m)$, where $n \leq p \leq n+(\lceil \frac{n}{2} \rceil -1)(n - \lceil \frac{n}{2} \rceil)$. There exists a flag-shaped blocker $B_n(m,t)$ with cardinality $p$ if and only if $p-n \leq m-1$ or if $p-n$ is a composite number.
\end{theorem}

\begin{proof}
    $\Rightarrow$ Suppose there exists a flag-shaped blocker with cardinality $p$. If $p-n$ is a prime number greater than $m-1$, then this implies that $t(n-m)$ is also prime and greater than $m-1$. However, $t(n-m)$ can only be prime if either $t=1$ or $n-m=1$. The requirement that $t(n-m) > m-1$ prevents either case from occurring. Thus, $p-n=t(n-m)$ cannot be prime, a contradiction. This implies that if there exists a flag-shaped blocker with cardinality $p$, then either $p-n \leq m-1$ is true or $p-n$ is a composite number.

    $\Leftarrow$ 
    For $n-m=1$, $p=n+t$ and $0 \leq t \leq m-1,$ so
    $$n \leq p \leq n+m-1$$
    $$0 \leq p-n \leq m-1,$$
    \noindent and all conditions on $p, n, m,$ and $t$ for a flag-shaped blocker, as defined in Section \ref{Section 2}, are met. This means that when $p-n \leq m-1$, there does exist a flag-shaped blocker with cardinality $p$.

    Now, suppose that $p-n$ is a composite number. If $n-m \geq 2$ and $t \geq 2$, then $p-n = t(n-m)$ must be a composite number such that $1 \leq m \leq n$ and $0 \leq t \leq m-1$. All conditions on $p, n, m,$ and $t$ for a flag-shaped blocker are satisfied, so there does exist a flag-shaped blocker with cardinality $p$. \end{proof}

Since we know the range of cardinalities for flag-shaped blockers, we briefly expand our view to consider all minimum blockers, not simply flag-shaped blockers. Because all minimal blockers are minimum blockers, the lower bound for the cardinality of a minimum blocker of an $n \times n$ matrix is $n$, as each letter of the Hankel-cyclic decomposition must appear at least once in the blocker. Determining the upper bound is more complicated, since the shapes of all minimum blockers have not yet been fully characterized. Nonetheless, we provide the following conjecture for the upper bound of the cardinality of all minimum blockers and show that flag-shaped blockers can achieve this upper bound.

\begin{conjecture} \label{conjecture1}
The upper bound for the cardinality of a minimum blocker of all $n \times n$ 
$123$-avoiding permutation matrices is $r \times s$, where $|r-s| \leq 1$ and $r + s = n + 1$.
\end{conjecture}

There exist minimum flag-shaped blockers with exactly $r \times s$ positions. These can occur when an adjacent set of positions forming a blocker given by the Frobenius-K\"{o}nig theorem, such that $|r-s| \leq 1$, is placed at the northwest or southeast corner of a matrix, since these blockers are special cases of flag-shaped blockers, as described in Section \ref{Section 2}. The upper bound can also be expressed in terms of only $m$ and $n$ by denoting $m=\lceil \frac{n}{2} \rceil$ and $t=m-1=\lceil \frac{n}{2} \rceil -1$ for a maximum cardinality of $n+(\lceil \frac{n}{2} \rceil -1)(n - \lceil \frac{n}{2} \rceil)$.

\begin{example} Minimum blockers of $10 \times 10$ matrices with cardinality $r \times s =30$.

\[\left[\begin{array}{c|c|c|c|c|c|c|c|c|c}
\cellcolor{red!20}a&\cellcolor{red!20}b&\cellcolor{red!20}c&\cellcolor{red!20}d&\cellcolor{red!20}e&\cellcolor{red!20}f&g&h&i&j\\ \hline
\cellcolor{red!20}b&\cellcolor{red!20}c&\cellcolor{red!20}d&\cellcolor{red!20}e&\cellcolor{red!20}f&\cellcolor{red!20}g&h&i&j&a\\ \hline
\cellcolor{red!20}c&\cellcolor{red!20}d&\cellcolor{red!20}e&\cellcolor{red!20}f&\cellcolor{red!20}g&\cellcolor{red!20}h&i&j&a&b\\ \hline
\cellcolor{red!20}d&\cellcolor{red!20}e&\cellcolor{red!20}f&\cellcolor{red!20}g&\cellcolor{red!20}h&\cellcolor{red!20}i&j&a&b&c\\ \hline
\cellcolor{red!20}e&\cellcolor{red!20}f&\cellcolor{red!20}g&\cellcolor{red!20}h&\cellcolor{red!20}i&\cellcolor{red!20}j&a&b&c&d\\ \hline
f&g&h&i&j&a&b&c&d&e\\ \hline
g&h&i&j&a&b&c&d&e&f\\ \hline
h&i&j&a&b&c&d&e&f&g\\ \hline
i&j&a&b&c&d&e&f&g&h\\ \hline
j&a&b&c&d&e&f&g&h&i
\end{array}\right]
\ \ and \ \
\left[\begin{array}{c|c|c|c|c|c|c|c|c|c}
a&b&c&d&e&f&g&h&i&j\\ \hline
b&c&d&e&f&g&h&i&j&a\\ \hline
c&d&e&f&g&h&i&j&a&b\\ \hline
d&e&f&g&h&i&j&a&b&c\\ \hline
e&f&g&h&i&\cellcolor{red!20}j&\cellcolor{red!20}a&\cellcolor{red!20}b&\cellcolor{red!20}c&\cellcolor{red!20}d\\ \hline
f&g&h&i&j&\cellcolor{red!20}a&\cellcolor{red!20}b&\cellcolor{red!20}c&\cellcolor{red!20}d&\cellcolor{red!20}e\\ \hline
g&h&i&j&a&\cellcolor{red!20}b&\cellcolor{red!20}c&\cellcolor{red!20}d&\cellcolor{red!20}e&\cellcolor{red!20}f\\ \hline
h&i&j&a&b&\cellcolor{red!20}c&\cellcolor{red!20}d&\cellcolor{red!20}e&\cellcolor{red!20}f&\cellcolor{red!20}g\\ \hline
i&j&a&b&c&\cellcolor{red!20}d&\cellcolor{red!20}e&\cellcolor{red!20}f&\cellcolor{red!20}g&\cellcolor{red!20}h\\ \hline
j&a&b&c&d&\cellcolor{red!20}e&\cellcolor{red!20}f&\cellcolor{red!20}g&\cellcolor{red!20}h&\cellcolor{red!20}i
\end{array}\right]\] \hfill $\cvd$
\end{example}


Furthermore, we are able to determine which cardinalities are unable to be obtained by flag-shaped blockers by using the sieve of Eratosthenes to find all values of $p-n$ that are prime, where $m-1 < p-n \leq rs-n$ and $|r-s| \leq 1$.

The last thing we will prove in this section is that Conjecture~\ref{conjecture1} is true for the $n=3$ case. That is, we can show that the upper bound of a minimum flag-shaped blocker of a $3 \times 3$ matrix is $r \times s = 2 \times 2 = 4$.

\begin{proof} Let $A$ be a $3\times 3$ $(0,1)$-matrix avoiding $123$-permutation.
If the permanent of $A$ is $0,$ then $A$ contains a rectangular $r\times s$ zero submatrix with $r+s=4.$ According to Frobenius-K\"{o}nig theorem, any zero not in the $r\times s$ submatrix is not necessary, so a minimum blocker of $A$ at most contains four zeros when $r=s=2.$

If the permanent of $A$ is nonzero, then $A$ contains three $1$'s on the diagonal.
\[\left[\begin{array}{c|c|c}
1&\phantom{0}&\phantom{0} \\ \hline
\phantom{0}&1&\phantom{0} \\ \hline
\phantom{0}&\phantom{0}&1
\end{array}\right]\]

If there is a pair of $1'$s symmetric with respect to the diagonal, we can construct a $123$-avoiding permutation matrix by replacing the two $1'$s on the diagonal by the pair of $1'$s.

\[\left[\begin{array}{c|c|c}
1&\phantom{0}&\cellcolor{red!20}\phantom{0} \\ \hline
\phantom{0}&1&\phantom{0} \\ \hline
\cellcolor{red!20}\phantom{0}&\phantom{0}&1
\end{array}\right] \Rightarrow \left[\begin{array}{c|c|c}
&\phantom{0}&\cellcolor{red!20}1 \\ \hline
\phantom{0}&1&\phantom{0} \\ \hline
\cellcolor{red!20}1&\phantom{0}&
\end{array}\right]\]

There are three pairs of elements and we just need to block one element in each pair, so a minimum blocker has cardinality $3.$ \end{proof}

\section{The Polytope Generated by 123-Avoiding Permutation Matrices}

 One of our motivations for studying blockers of $123$-avoiding permutation matrices is to gain a better understanding of the polytope generated by $123$-avoiding permutation matrices, denoted $\Omega_n(\overline{123})$, whose dimension is $(n-1)^2$. In \cite{BC3}, Brualdi and Cao show that each minimal blocker of $n\times n$ $123$-avoiding permutation matrices determines a facet of $\Omega_n(\overline{123}),$ and a face of $\Omega_n(\overline{123})$ lives in dimension $(n-1)^2-1.$ We seek to learn more about the dimension of the faces of $\Omega_n(\overline{123})$ determined by flag-shaped blocker. The dimension of a face of $\Omega_n(\overline{123})$ determined by a minimum blocker is equivalent to the number of linearly independent $123$-avoiding permutation matrices that intersect the blocker exactly once, and no more. We present an inductive argument to find a more precise upper bound for the dimension.

\begin{lemma}\label{Lemma18} Given a flag-shaped minimum blocker $B_n(m,t)$ of all $n \times n$ 
$123$-avoiding permutation matrices with $m=n-1$, there is a $t \times 1$ submatrix of adjacent positions containing the $(n,n)$ position that cannot be used to construct a $123$-avoiding permutation matrix that intersects the blocker exactly once.
\end{lemma}

\begin{proof}
Consider an $n \times n$ matrix containing a flag-shaped blocker such that $m=n-1$. Our claim is that the elements at the intersection of column $n$ and the last $t$ rows cannot be used to construct any $123$-avoiding permutation matrix that intersects the blocker exactly once. 

Suppose we use any position from the intersection of column $n$ and the last $t$ rows to attempt to construct a $123$-avoiding permutation matrix that intersects the blocker exactly once. Now, consider the submatrix formed by deleting the column and row in which this position resides. We must utilize $t-1$ additional positions from the last $t-1$ rows of the submatrix to attempt to construct a $123$-avoiding permutation matrix. The specific positions we use are not important in this case, so long as they do not form a $123$-pattern with the first position we chose from column $n$.

We can now delete the last $t-1$ rows from the submatrix. Notice that the new submatrix is not square, as it has $t-1$ more columns than rows. No matter which $t-1$ columns we delete to form an $(n-t) \times (n-t)$ submatrix, we can never construct a $12$-avoiding permutation matrix in this submatrix without intersecting the blocker more than once. This is due to the fact that the remaining flag portion of the blocker will always form at least an $(n-m+1) \times 2$ submatrix with at least two blocker positions on each full diagonal of the $(n-t) \times (n-1)$ submatrix. Thus, this $(n-t) \times (n-1)$ submatrix must contain a $12$-pattern. However, this $12$-pattern paired with the position from column $n$ forms a $123$-pattern. Thus, we cannot construct a $123$-avoiding permutation matrix using any position from the intersection of column $n$ and the last $t$ rows.
\end{proof}

\begin{theorem}\label{Theorem19}
Given a flag-shaped minimum blocker $B_n(m,t)$ of all $n \times n$ $123$-avoiding permutation matrices, 
there is a $t \times (n-m)$ submatrix of adjacent positions containing the $(n,n)$ position that cannot be used to construct a $123$-avoiding permutation matrix that intersects the blocker exactly once. \end{theorem}

\vspace{-0.2in}
\begin{proof} Define $p:=n-m$. Lemma \ref{Lemma18} describes the case where $p=1$. Now, suppose that when $2\leq p\leq k$, there are $pt$ positions that cannot be used to construct a $123$-avoiding permutation matrix that intersects the blocker exactly once.

Suppose $p=k+1$. Consider any of the $t$ positions of an $n \times n$ matrix in the intersection of the $n$th column and the last $t$ rows. In order to construct a $123$-avoiding permutation matrix using any of these positions, a necessary condition is that we can construct a $123$-avoiding permutation matrix in the submatrix obtained by deleting the row and column the position that intersects the blocker only once resides in.

However, the $(n-1) \times (n-1)$ submatrix we obtain contains a flag-shaped blocker (perhaps with more blocker positions than necessary) with $p=k$. Then by the induction hypothesis, there are $kt$ positions we cannot use to construct a $123$-avoiding permutation matrix in the $(n-1) \times (n-1)$ submatrix  that intersects the blocker only once. We also know there are $t$ positions from the intersection of the last column and the last $t$ rows that we cannot use to construct a $123$-avoiding permutation matrix. Thus, there are a total of $kt+t=(k+1)t$ positions we cannot use to construct a $123$-avoiding permutation matrix that intersects the blocker only once. Then by mathematical induction, we have shown that there are $pt=(n-m)t$ positions of an $n\times n$ matrix that we cannot use. Thus, there is a $t \times (n-m)$ submatrix at the lower right corner of the matrix from which we cannot use any positions to construct a $123$-avoiding permutation matrix that intersects the blocker exactly once. 
\end{proof}

\begin{example}
To illustrate Theorem \ref{Theorem19}, consider the flag-shaped blocker $B_{10}(8,3)$. We are unable to construct a $123$-avoiding permutation matrix that intersects the blocker once, and no more, using any of the yellow positions.

$$\left[\begin{array}{c|c|c|c|c|c|c|c|c|c}
a&b&c&d&\cellcolor{red!20}e&\cellcolor{red!20}f&\cellcolor{red!20}g&\cellcolor{red!20}h&i&j\\ \hline
b&c&d&e&\cellcolor{red!20}f&\cellcolor{red!20}g&\cellcolor{red!20}h&\cellcolor{red!20}i&j&a\\ \hline
c&d&e&f&\cellcolor{red!20}g&\cellcolor{red!20}h&\cellcolor{red!20}i&\cellcolor{red!20}j&a&b\\ \hline
d&e&f&g&h&i&j&\cellcolor{red!20}a&b&c\\ \hline
e&f&g&h&i&j&a&\cellcolor{red!20}b&c&d\\ \hline
f&g&h&i&j&a&b&\cellcolor{red!20}c&d&e\\ \hline
g&h&i&j&a&b&c&\cellcolor{red!20}d&e&f\\ \hline
h&i&j&a&b&c&d&e&\cellcolor{yellow!20}f&\cellcolor{yellow!20}g\\ \hline
i&j&a&b&c&d&e&f&\cellcolor{yellow!20}g&\cellcolor{yellow!20}h\\ \hline
j&a&b&c&d&e&f&g&\cellcolor{yellow!20}h&\cellcolor{yellow!20}i
\end{array}\right]$$

The permutation matrix consisting of all $i$'s is the only $123$-avoiding permutation matrix using the yellow $i$. Clearly, this intersects the blocker twice. Thus, we can focus on just the yellow $f$ or one of the yellow $g$'s or $h$'s to illustrate the issue that arises when attempting to use any of these positions to form a $123$-avoiding permutation matrix. Without loss of generality, we will consider the $h$ in the last column.

We can delete column $10$ and row $9$ since we do not have to use another position from either in our construction of a $123$-avoiding permutation matrix. We also know that we will use one position from row $8$ and one from row $10$ for the permutation matrix, so we can focus on the $7 \times 9$ submatrix formed by deleting the last $t=3$ rows and the last column.

$$\left[\begin{array}{c|c|c|c|c|c|c|c|c||c}
a&b&c&d&\cellcolor{red!20}e&\cellcolor{red!20}f&\cellcolor{red!20}g&\cellcolor{red!20}h&i&j\\ \hline
b&c&d&e&\cellcolor{red!20}f&\cellcolor{red!20}g&\cellcolor{red!20}h&\cellcolor{red!20}i&j&a\\ \hline
c&d&e&f&\cellcolor{red!20}g&\cellcolor{red!20}h&\cellcolor{red!20}i&\cellcolor{red!20}j&a&b\\ \hline
d&e&f&g&h&i&j&\cellcolor{red!20}a&b&c\\ \hline
e&f&g&h&i&j&a&\cellcolor{red!20}b&c&d\\ \hline
f&g&h&i&j&a&b&\cellcolor{red!20}c&d&e\\ \hline
g&h&i&j&a&b&c&\cellcolor{red!20}d&e&f\\ \hline \hline
h&i&j&a&b&c&d&e&f&g\\ \hline
i&j&a&b&c&d&e&f&g&\cellcolor{green!20}h\\ \hline
j&a&b&c&d&e&f&g&h&i
\end{array}\right]$$

We can delete the two columns in which the positions we use from rows $8$ and $10$ reside. However, notice that there are no two columns we can delete to form a $7 \times 7$ submatrix in which we can construct a $12$-avoiding permutation matrix that intersects the blocker exactly once. For example, suppose we utilize the $e$ in row $8$ and the $f$ in row $10$ for the permutation matrix. The submatrix formed by deleting all rows and columns that already contain a position of the $123$-avoiding permutation matrix we are attempting to construct follows.

$$\left[\begin{array}{c|c|c|c|c|c|c}
a&b&c&d&\cellcolor{red!20}e&\cellcolor{red!20}f&g\\ \hline
b&c&d&e&\cellcolor{red!20}f&\cellcolor{red!20}g&a\\ \hline
c&d&e&f&\cellcolor{red!20}g&\cellcolor{red!20}a&b\\ \hline
d&e&f&g&a&b&c\\ \hline
e&f&g&a&b&c&d\\ \hline
f&g&a&b&c&d&e\\ \hline
g&a&b&c&d&e&f
\end{array}\right]$$

The $h$ from row $9$ of the original matrix is completely below and to the right of this $7 \times 7$ submatrix, so the only way we can complete the $n \times n$ $123$-avoiding permutation matrix is to construct a $12$-avoiding permutation matrix in this submatrix. However, this is only possible if we take the Hankel diagonal of the submatrix, but we intersect the blocker twice by doing so. Thus, it is impossible to construct a $123$-avoiding permutation matrix using the $h$ in the $(9,10)$ position without intersecting the blocker more than once. The same process can be used to show that none of the remaining positions in the $t \times (n-m)$ submatrix at the lower right corner of the matrix can be used to construct a $123$-avoiding permutation matrix intersecting the blocker only once. \hfill $\cvd$
\end{example}




Using Theorem \ref{Theorem19}, we can determine the upper bound for the dimension of a face of $\Omega_n(\overline{123})$.

\begin{theorem}\label{Theorem11} A flag-shaped minimum blocker $B_n(m,t)$ of all $n \times n$ $123$-avoiding permutation matrices  
determines a face of $\Omega_n(\overline{123})$ with dimensions at most $(n-1)^2+1-t(n-m)$.
\end{theorem}

\begin{proof}
Suppose there are $a$ linearly independent permutation matrices that intersect the blocker exactly once and that do not use any positions from the $t \times (n-m)$ submatrix at the lower right corner below and to the right of the flag-shaped blocker. It is possible to find $t(n-m)$ linearly independent permutation matrices (which do not intersect the blocker exactly once according to Theorem \ref{Theorem19}), each of which uses a unique position of the $t \times (n-1)$ submatrix. These $t(n-m)$ permutation matrices are linearly independent from one another and from the $a$ linearly independent permutation matrices found earlier. Thus, we have a total of $a+t(n-m)$ linearly independent permutation matrices, and this total cannot exceed the total number of linearly independent permutation matrices of an $n \times n$ matrix, which is $(n-1)^2+1$. Therefore, the number of linearly independent permutation matrices that intersects the blocker exactly once, given by $a$, is
$$a \leq (n-1)^2+1 - t(n-m).$$ \end{proof}

We have shown that there are at least $t(n-m)$ positions of an $n \times n$ matrix that we cannot use to construct a $123$-avoiding permutation matrix that intersects a minimum flag-shaped blocker exactly once. 
It is important to note that not all flag-shaped blockers will define a face of $\Omega_n(\overline{123})$ with dimension $(n-1)^2+1-t(n-m)$. However, there exist flag-shaped blockers that do define a face with precisely this dimension. Before describing these blockers, we first note an important aspect of of row and column blockers.

\begin{lemma} \label{Lemma12}
For each minimum blocker of all $n \times n$ $123$-avoiding permutation matrices that are composed of an entire row or column of an $n \times n$ matrix, there exist $(n-1)^2+1$ linearly independent $123$-avoiding permutation matrices that intersect the blocker exactly once each.
\end{lemma}

\begin{proof}
Every $n \times n$ permutation matrix must intersect a row or column blocker exactly once. The polytope $\Omega_n(\overline{123})$ lives in dimension $(n-1)^2+1$, so there must exist the same number of linearly independent $123$-avoiding permutation matrices.
\end{proof}

\begin{theorem}
A rectangular flag-shaped blocker $B_n(m,m-1)$ of all $n \times n$ $123$-avoiding permutation matrices determines a face of $\Omega_n(\overline{123})$ with the maximum dimension $(n-1)^2+1-t(n-m)$.
\end{theorem}

\begin{proof}
We use induction on $p:=n-m$, starting by showing that when $p=1$, it is possible to find $(n-1)^2+1-t \cdot 1$ linearly independent $123$-avoiding permutation matrices that intersect the blocker exactly once. For illustrative purposes, we consider a $10 \times 10$ matrix with the flag-shaped blocker $B_n(9,8)$ while describing the general proof. Additionally, notice that $t=m-1$ for all rectangular flag-shaped blockers, and $m=n-1$ when $p=1$, so we are looking to construct $(n-1)^2+1-(n-2)\cdot 1$ linearly independent $123$-avoiding permutation matrices.

Using the $(1,n)$ position, there are $(n-2)^2+1$ linearly independent $123$-avoiding permutation matrices that intersect the blocker exactly once. This is because the $(n-1) \times (n-1)$ submatrix obtained by deleting the first row and last column contains $[(n-1)-1]^2+1$ such permutation matrices according to Lemma \ref{Lemma12}.

$$\left[\begin{array}{c|c|c|c|c|c|c|c|c||c}
\cellcolor{red!20}a&\cellcolor{red!20}b&\cellcolor{red!20}c&\cellcolor{red!20}d&\cellcolor{red!20}e&\cellcolor{red!20}f&\cellcolor{red!20}g&\cellcolor{red!20}h&\cellcolor{red!20}i&\cellcolor{green!20}j\\ \hline \hline
\cellcolor{red!20}b&\cellcolor{red!20}c&\cellcolor{red!20}d&\cellcolor{red!20}e&\cellcolor{red!20}f&\cellcolor{red!20}g&\cellcolor{red!20}h&\cellcolor{red!20}i&\cellcolor{red!20}j&a\\ \hline
c&d&e&f&g&h&i&j&a&b\\ \hline
d&e&f&g&h&i&j&a&b&c\\ \hline
e&f&g&h&i&j&a&b&c&d\\ \hline
f&g&h&i&j&a&b&c&d&e\\ \hline
g&h&i&j&a&b&c&d&e&f\\ \hline
h&i&j&a&b&c&d&e&f&g\\ \hline
i&j&a&b&c&d&e&f&g&h\\ \hline
j&a&b&c&d&e&f&g&h&i
\end{array}\right]$$

Using the $(2,n)$ position, there are an additional $n-1$ linearly independent $123$-avoiding permutation matrices that intersect the blocker exactly once, which are obtained using the $n-1$ blocker positions in the first row. The $123$-avoiding permutation matrices corresponding to each position will contain a unique position of the matrix, making them linearly independent from one another. For example, consider the $123$-avoiding permutation matrix obtained by using the yellow $e$ from the blocker in the example below.

$$\left[\begin{array}{c|c|c|c|c|c|c|c|c|c}
\cellcolor{red!20}a&\cellcolor{red!20}b&\cellcolor{red!20}c&\cellcolor{red!20}d&\cellcolor{yellow!20}e&\cellcolor{red!20}f&\cellcolor{red!20}g&\cellcolor{red!20}h&\cellcolor{red!20}i&j\\ \hline 
\cellcolor{red!20}b&\cellcolor{red!20}c&\cellcolor{red!20}d&\cellcolor{red!20}e&\cellcolor{red!20}f&\cellcolor{red!20}g&\cellcolor{red!20}h&\cellcolor{red!20}i&\cellcolor{red!20}j&\cellcolor{green!20}a\\ \hline
c&d&e&f&g&h&i&j&\cellcolor{green!20}a&b\\ \hline
d&e&f&g&h&i&j&\cellcolor{green!20}a&b&c\\ \hline
e&f&g&h&i&j&\cellcolor{green!20}a&b&c&d\\ \hline
f&g&h&i&j&\cellcolor{green!20}a&b&c&d&e\\ \hline
g&h&i&\cellcolor{green!20}j&a&b&c&d&e&f\\ \hline
h&i&\cellcolor{green!20}j&a&b&c&d&e&f&g\\ \hline
i&\cellcolor{green!20}j&a&b&c&d&e&f&g&h\\ \hline
\cellcolor{green!20}j&a&b&c&d&e&f&g&h&i
\end{array}\right]$$

Thus, in total we have 
$$(n-2)^2+1 + n-1 = (n-1)^2+1-(n-2)\cdot 1$$
linearly independent $123$-avoiding permutation matrices that intersect the blocker exactly once, as desired.

We move on to the induction step. Suppose that when $p=k$, where $2 \leq k \leq n-3$, the corresponding rectangular flag-shaped blocker achieves the maximum dimension 
$$(n-1)^2+1-(m-1)(k)=(n-1)^2+1-(m-1)(n-m-1),$$
since $t=m-1$ and $k=n-m-1$ for these blockers. We will show that if $p=k+1$, the rectangular flag-shaped blocker achieves the maximum dimension 
$$(n-1)^2+1-(m-1)(k+1)=(n-1)^2+1-(m-1)(n-m).$$ 
For illustrative purposes, we consider a $10 \times 10$ matrix with the flag-shaped blocker $B_n(5,4)$ while describing the general proof.

Using the inductive hypothesis, it is possible to find 
$$[(n-1)-1]^2+1-(m-1)[(n-1)-m] = (n-2)^2+1-(m-1)(n-m-1)$$
linearly independent $123$-avoiding permutation matrices that intersect the blocker exactly once using the $(1,n)$ position.

$$\left[\begin{array}{c|c|c|c|c|c|c|c|c||c}
\cellcolor{red!20}a&\cellcolor{red!20}b&\cellcolor{red!20}c&\cellcolor{red!20}d&\cellcolor{red!20}e&f&g&h&i&\cellcolor{green!20}j\\ \hline \hline
\cellcolor{red!20}b&\cellcolor{red!20}c&\cellcolor{red!20}d&\cellcolor{red!20}e&\cellcolor{red!20}f&g&h&i&j&a\\ \hline
\cellcolor{red!20}c&\cellcolor{red!20}d&\cellcolor{red!20}e&\cellcolor{red!20}f&\cellcolor{red!20}g&h&i&j&a&b\\ \hline
\cellcolor{red!20}d&\cellcolor{red!20}e&\cellcolor{red!20}f&\cellcolor{red!20}g&\cellcolor{red!20}h&i&j&a&b&c\\ \hline
\cellcolor{red!20}e&\cellcolor{red!20}f&\cellcolor{red!20}g&\cellcolor{red!20}h&\cellcolor{red!20}i&j&a&b&c&d\\ \hline
\cellcolor{red!20}f&\cellcolor{red!20}g&\cellcolor{red!20}h&\cellcolor{red!20}i&\cellcolor{red!20}j&a&b&c&d&e\\ \hline
g&h&i&j&a&b&c&d&e&f\\ \hline
h&i&j&a&b&c&d&e&f&g\\ \hline
i&j&a&b&c&d&e&f&g&h\\ \hline
j&a&b&c&d&e&f&g&h&i
\end{array}\right]$$

Using the $(2,n)$ position and the blocker positions in the first row, we can find $m$ additional linearly independent $123$-avoiding permutation matrices that intersect the blocker exactly once and that use a unique position of the matrix. For example, consider such a permutation matrix utilizing the yellow $d$ from the first row of the below matrix.

$$\left[\begin{array}{c|c|c|c|c|c|c|c|c|c}
\cellcolor{red!20}a&\cellcolor{red!20}b&\cellcolor{red!20}c&\cellcolor{yellow!20}d&\cellcolor{red!20}e&f&g&h&i&j\\ \hline 
\cellcolor{red!20}b&\cellcolor{red!20}c&\cellcolor{red!20}d&\cellcolor{red!20}e&\cellcolor{red!20}f&g&h&i&j&\cellcolor{green!20}a\\ \hline
\cellcolor{red!20}c&\cellcolor{red!20}d&\cellcolor{red!20}e&\cellcolor{red!20}f&\cellcolor{red!20}g&h&i&j&\cellcolor{green!20}a&b\\ \hline
\cellcolor{red!20}d&\cellcolor{red!20}e&\cellcolor{red!20}f&\cellcolor{red!20}g&\cellcolor{red!20}h&i&j&\cellcolor{green!20}a&b&c\\ \hline
\cellcolor{red!20}e&\cellcolor{red!20}f&\cellcolor{red!20}g&\cellcolor{red!20}h&\cellcolor{red!20}i&j&\cellcolor{green!20}a&b&c&d\\ \hline
\cellcolor{red!20}f&\cellcolor{red!20}g&\cellcolor{red!20}h&\cellcolor{red!20}i&\cellcolor{red!20}j&\cellcolor{green!20}a&b&c&d&e\\ \hline
g&h&i&j&\cellcolor{green!20}a&b&c&d&e&f\\ \hline
h&i&\cellcolor{green!20}j&a&b&c&d&e&f&g\\ \hline
i&\cellcolor{green!20}j&a&b&c&d&e&f&g&h\\ \hline
\cellcolor{green!20}j&a&b&c&d&e&f&g&h&i
\end{array}\right]$$

Furthermore, the $n-m-1$ unused positions to the right of the blocker in the first row can also be used to construct linearly independent $123$-avoiding permutation matrices that intersect the blocker exactly once and that use a unique position of the matrix. For example, consider such a permutation matrix utilizing the green $g$ from the first row below, where the yellow represents the intersection of the blocker and the $123$-avoiding permutation matrix.

$$\left[\begin{array}{c|c|c|c|c|c|c|c|c|c}
\cellcolor{red!20}a&\cellcolor{red!20}b&\cellcolor{red!20}c&\cellcolor{red!20}d&\cellcolor{red!20}e&f&\cellcolor{green!20}g&h&i&j\\ \hline
\cellcolor{red!20}b&\cellcolor{red!20}c&\cellcolor{red!20}d&\cellcolor{red!20}e&\cellcolor{red!20}f&g&h&i&j&\cellcolor{green!20}a\\ \hline
\cellcolor{red!20}c&\cellcolor{red!20}d&\cellcolor{red!20}e&\cellcolor{red!20}f&\cellcolor{red!20}g&h&i&j&\cellcolor{green!20}a&b\\ \hline
\cellcolor{red!20}d&\cellcolor{red!20}e&\cellcolor{red!20}f&\cellcolor{red!20}g&\cellcolor{red!20}h&i&j&\cellcolor{green!20}a&b&c\\ \hline
\cellcolor{red!20}e&\cellcolor{red!20}f&\cellcolor{red!20}g&\cellcolor{red!20}h&\cellcolor{red!20}i&\cellcolor{green!20}j&a&b&c&d\\ \hline
\cellcolor{red!20}f&\cellcolor{red!20}g&\cellcolor{red!20}h&\cellcolor{red!20}i&\cellcolor{yellow!20}j&a&b&c&d&e\\ \hline
g&h&i&\cellcolor{green!20}j&a&b&c&d&e&f\\ \hline
h&i&\cellcolor{green!20}j&a&b&c&d&e&f&g\\ \hline
i&\cellcolor{green!20}j&a&b&c&d&e&f&g&h\\ \hline
\cellcolor{green!20}j&a&b&c&d&e&f&g&h&i
\end{array}\right]$$

Similarly, there are $n-t-2=n-m-1$ positions in the intersection of the last column and row $3$ through row $n-m+1$ that can be used to construct linearly independent $123$-avoiding permutation matrices that intersect the blocker exactly once and that use a unique position of the matrix. For example, consider such a permutation matrix utilizing the green $c$ from the last column below.

$$\left[\begin{array}{c|c|c|c|c|c|c|c|c|c}
\cellcolor{red!20}a&\cellcolor{red!20}b&\cellcolor{red!20}c&\cellcolor{red!20}d&\cellcolor{red!20}e&f&g&h&\cellcolor{green!20}i&j\\ \hline 
\cellcolor{red!20}b&\cellcolor{red!20}c&\cellcolor{red!20}d&\cellcolor{red!20}e&\cellcolor{red!20}f&g&h&\cellcolor{green!20}i&j&a\\ \hline
\cellcolor{red!20}c&\cellcolor{red!20}d&\cellcolor{red!20}e&\cellcolor{red!20}f&\cellcolor{red!20}g&h&\cellcolor{green!20}i&j&a&b\\ \hline
\cellcolor{red!20}d&\cellcolor{red!20}e&\cellcolor{red!20}f&\cellcolor{red!20}g&\cellcolor{red!20}h&i&j&a&b&\cellcolor{green!20}c\\ \hline
\cellcolor{red!20}e&\cellcolor{red!20}f&\cellcolor{red!20}g&\cellcolor{red!20}h&\cellcolor{red!20}i&\cellcolor{green!20}j&a&b&c&d\\ \hline
\cellcolor{red!20}f&\cellcolor{red!20}g&\cellcolor{red!20}h&\cellcolor{red!20}i&\cellcolor{yellow!20}j&a&b&c&d&e\\ \hline
g&h&i&\cellcolor{green!20}j&a&b&c&d&e&f\\ \hline
h&i&\cellcolor{green!20}j&a&b&c&d&e&f&g\\ \hline
i&\cellcolor{green!20}j&a&b&c&d&e&f&g&h\\ \hline
\cellcolor{green!20}j&a&b&c&d&e&f&g&h&i
\end{array}\right]$$

Thus, in total we have found
$$(n-2)^2+1-(m-1)(n-m-1) + m + 2(n-m-1) = (n-1)^2+1-(m-1)(n-m)$$
linearly independent $123$-avoiding permutation matrices that intersect the blocker exactly once, as desired. 
\end{proof}

While Theorem \ref{Theorem11} presents the upper bound of the dimension of a face of $\Omega_n(\overline{123})$ as determined by a flag-shaped minimum blocker, it does not establish the lower bound. We propose one lower bound, though we first consider a theorem regarding $L$-shaped blockers, a special type of flag-shaped blocker.

\begin{theorem} [Theorem 3.7 in \cite{BC3}]
A minimum blocker of all $n \times n$ $123$-avoiding permutation matrices determines a facet of the polytope $\Omega_n(\overline{123})$ whose extreme points are the $n \times n$ $123$-avoiding permutation matrices that intersect the blocker exactly once.
\end{theorem}

\begin{corollary} \label{Lemma 14}
For $L$-shaped minimum blockers of all $n \times n$ $123$-avoiding permutation matrices, there exist $(n-1)^2$ linearly independent $n \times n$ $123$-avoiding permutation matrices that each contain exactly one blocker position.
\end{corollary}

We can utilize Corollary \ref{Lemma 14} to determine a lower bound for the dimension of a face of $\Omega_n(\overline{123})$.

\begin{theorem}
A flag-shaped minimum blocker $B_n(m,t)$ of all $n \times n$ $123$-avoiding permutation matrices 
determines a face of $\Omega_n(\overline{123})$ with dimension at least $(n-1)^2+1-(t+2)(n-m)$.
\end{theorem}

\begin{proof}
Inducting on $p:=n-m$, we first show that when $p=1$, it is possible to find $(n-1)^2+1-(t+2)$ linearly independent $123$-avoiding permutation matrices that intersect the blocker exactly once. For illustrative purposes, we consider a $10 \times 10$ matrix with the flag-shaped blocker $B_n(9,3)$ while describing the general proof.

Using the $(1,n)$ position, there are $(n-2)^2$ linearly independent $123$-avoiding permutation matrices that intersect the blocker exactly once. This is because the $(n-1) \times (n-1)$ submatrix obtained by deleting the first row and last column contains $[(n-1)-1]^2$ such permutation matrices according to Lemma \ref{Lemma 14}.

$$\left[\begin{array}{c|c|c|c|c|c|c|c|c||c}
a&b&c&d&e&\cellcolor{red!20}f&\cellcolor{red!20}g&\cellcolor{red!20}h&\cellcolor{red!20}i&\cellcolor{green!20}j\\ \hline \hline
b&c&d&e&f&\cellcolor{red!20}g&\cellcolor{red!20}h&\cellcolor{red!20}i&\cellcolor{red!20}j&a\\ \hline
c&d&e&f&g&h&i&j&\cellcolor{red!20}a&b\\ \hline
d&e&f&g&h&i&j&a&\cellcolor{red!20}b&c\\ \hline
e&f&g&h&i&j&a&b&\cellcolor{red!20}c&d\\ \hline
f&g&h&i&j&a&b&c&\cellcolor{red!20}d&e\\ \hline
g&h&i&j&a&b&c&d&\cellcolor{red!20}e&f\\ \hline
h&i&j&a&b&c&d&e&f&g\\ \hline
i&j&a&b&c&d&e&f&g&h\\ \hline
j&a&b&c&d&e&f&g&h&i
\end{array}\right]$$

Using $(2,n)$ in combination with the first $m-t-1$ positions from row $1$, in addition to $(n-t, n)$ in combination with the blocker positions in row $1$, we have an additional $m=n-1$ linearly independent $123$-avoiding permutation matrices that intersect the blocker once and that use a unique position of the matrix. Consider the following two matrices as examples, where yellow positions represent the intersection of the permutation matrix with the blocker.

$$\left[\begin{array}{c|c|c|c|c|c|c|c|c|c}
a&b&\cellcolor{green!20}c&d&e&\cellcolor{red!20}f&\cellcolor{red!20}g&\cellcolor{red!20}h&\cellcolor{red!20}i&j\\ \hline
b&c&d&e&f&\cellcolor{red!20}g&\cellcolor{red!20}h&\cellcolor{red!20}i&\cellcolor{red!20}j&\cellcolor{green!20}a\\ \hline
c&d&e&f&g&h&i&j&\cellcolor{yellow!20}a&b\\ \hline
d&e&f&g&h&i&j&\cellcolor{green!20}a&\cellcolor{red!20}b&c\\ \hline
e&f&g&h&i&j&\cellcolor{green!20}a&b&\cellcolor{red!20}c&d\\ \hline
f&g&h&i&j&\cellcolor{green!20}a&b&c&\cellcolor{red!20}d&e\\ \hline
g&h&i&j&\cellcolor{green!20}a&b&c&d&\cellcolor{red!20}e&f\\ \hline
h&i&j&\cellcolor{green!20}a&b&c&d&e&f&g\\ \hline
i&\cellcolor{green!20}j&a&b&c&d&e&f&g&h\\ \hline
\cellcolor{green!20}j&a&b&c&d&e&f&g&h&i
\end{array}\right] \text{ \ \ and \ \ }
\left[\begin{array}{c|c|c|c|c|c|c|c|c|c}
a&b&c&d&e&\cellcolor{red!20}f&\cellcolor{red!20}g&\cellcolor{yellow!20}h&\cellcolor{red!20}i&j\\ \hline
b&c&d&e&\cellcolor{green!20}f&\cellcolor{red!20}g&\cellcolor{red!20}h&\cellcolor{red!20}i&\cellcolor{red!20}j&a\\ \hline
c&d&e&\cellcolor{green!20}f&g&h&i&j&\cellcolor{red!20}a&b\\ \hline
d&e&\cellcolor{green!20}f&g&h&i&j&a&\cellcolor{red!20}b&c\\ \hline
e&\cellcolor{green!20}f&g&h&i&j&a&b&\cellcolor{red!20}c&d\\ \hline
\cellcolor{green!20}f&g&h&i&j&a&b&c&\cellcolor{red!20}d&e\\ \hline
g&h&i&j&a&b&c&d&\cellcolor{red!20}e&\cellcolor{green!20}f\\ \hline
h&i&j&a&b&c&d&e&\cellcolor{green!20}f&g\\ \hline
i&j&a&b&c&d&\cellcolor{green!20}e&f&g&h\\ \hline
j&a&b&c&d&\cellcolor{green!20}e&f&g&h&i
\end{array}\right]$$

Additionally, we can construct $n-t-3$ more linearly independent $123$-avoiding permutation matrices that intersect the blocker only once by utilizing the positions at the intersection of column $n$ and rows $3$ through $n-t-1$. Consider the following example.

$$\left[\begin{array}{c|c|c|c|c|c|c|c|c|c}
a&b&c&d&e&\cellcolor{red!20}f&\cellcolor{red!20}g&\cellcolor{red!20}h&\cellcolor{yellow!20}i&j\\ \hline
b&c&d&e&\cellcolor{green!20}f&\cellcolor{red!20}g&\cellcolor{red!20}h&\cellcolor{red!20}i&\cellcolor{red!20}j&a\\ \hline
c&d&e&\cellcolor{green!20}f&g&h&i&j&\cellcolor{red!20}a&b\\ \hline
d&e&f&g&h&i&j&a&\cellcolor{red!20}b&\cellcolor{green!20}c\\ \hline
e&f&\cellcolor{green!20}g&h&i&j&a&b&\cellcolor{red!20}c&d\\ \hline
f&\cellcolor{green!20}g&h&i&j&a&b&c&\cellcolor{red!20}d&e\\ \hline
\cellcolor{green!20}g&h&i&j&a&b&c&d&\cellcolor{red!20}e&f\\ \hline
h&i&j&a&b&c&d&\cellcolor{green!20}e&f&g\\ \hline
i&j&a&b&c&d&\cellcolor{green!20}e&f&g&h\\ \hline
j&a&b&c&d&\cellcolor{green!20}e&f&g&h&i
\end{array}\right]$$

In total we have 
$$(n-2)^2 + n-1 + n-t-3= (n-1)^2 +1 - (t+2)$$
linearly independent $123$-avoiding permutation matrices that intersect the blocker exactly once, as desired.

Moving onto the induction step, suppose that when $p=k$, where $2 \leq k \leq n-3$, the corresponding flag-shaped blocker lives in dimension at least
$$(n-1)^2+1-(t+2)(k).$$
We will show that if $p=k+1$, the flag-shaped blocker at least achieves the minimum dimension 
$$(n-1)^2+1-(t+2)(k+1).$$ 
For illustrative purposes, we consider a $10 \times 10$ matrix with the flag-shaped blocker $B_n(7,2)$ while describing the general proof.

Using the inductive hypothesis, it is possible to find 
$$[(n-1)-1)]^2+1-(t+2)(k) = (n-2)^2+1-(t+2)(n-m)$$
linearly independent $123$-avoiding permutation matrices that intersect the blocker exactly once using the $(1,n)$ position.

$$\left[\begin{array}{c|c|c|c|c|c|c|c|c||c}
a&b&c&d&\cellcolor{red!20}e&\cellcolor{red!20}f&\cellcolor{red!20}g&h&i&\cellcolor{green!20}j\\ \hline \hline
b&c&d&e&\cellcolor{red!20}f&\cellcolor{red!20}g&\cellcolor{red!20}h&i&j&a\\ \hline
c&d&e&f&\cellcolor{red!20}g&\cellcolor{red!20}h&\cellcolor{red!20}i&j&a&b\\ \hline
d&e&f&g&\cellcolor{red!20}h&\cellcolor{red!20}i&\cellcolor{red!20}j&a&b&c\\ \hline
e&f&g&h&i&j&\cellcolor{red!20}a&b&c&d\\ \hline
f&g&h&i&j&a&\cellcolor{red!20}b&c&d&e\\ \hline
g&h&i&j&a&b&\cellcolor{red!20}c&d&e&f\\ \hline
h&i&j&a&b&c&\cellcolor{red!20}d&e&f&g\\ \hline
i&j&a&b&c&d&e&f&g&h\\ \hline
j&a&b&c&d&e&f&g&h&i
\end{array}\right]$$

Utilizing the first $m-t-1$ positions of row $1$ in conjunction with the $(2,n)$ position, we can construct $m-t-1$ additional linearly independent $123$-avoiding permutation matrices intersecting the blocker only once and containing a unique position of the matrix. Consider the below example using $c$ from row $1$ along with $a$ from row $2$.

$$\left[\begin{array}{c|c|c|c|c|c|c|c|c|c}
a&b&\cellcolor{green!20}c&d&\cellcolor{red!20}e&\cellcolor{red!20}f&\cellcolor{red!20}g&h&i&j\\ \hline
b&c&d&e&\cellcolor{red!20}f&\cellcolor{red!20}g&\cellcolor{red!20}h&i&j&\cellcolor{green!20}a\\ \hline
c&d&e&f&\cellcolor{red!20}g&\cellcolor{red!20}h&\cellcolor{red!20}i&j&\cellcolor{green!20}a&b\\ \hline
d&e&f&g&\cellcolor{red!20}h&\cellcolor{red!20}i&\cellcolor{red!20}j&\cellcolor{green!20}a&b&c\\ \hline
e&f&g&h&i&j&\cellcolor{yellow!20}a&b&c&d\\ \hline
f&g&h&i&j&\cellcolor{green!20}a&\cellcolor{red!20}b&c&d&e\\ \hline
g&h&i&j&\cellcolor{green!20}a&b&\cellcolor{red!20}c&d&e&f\\ \hline
h&i&j&\cellcolor{green!20}a&b&c&\cellcolor{red!20}d&e&f&g\\ \hline
i&\cellcolor{green!20}j&a&b&c&d&e&f&g&h\\ \hline
\cellcolor{green!20}j&a&b&c&d&e&f&g&h&i
\end{array}\right]$$

Furthermore, we can use the row $1$ blocker positions along with the $(m-t+1,n)$ position to obtain $t+1$ additional linearly independent $123$-avoiding permutation matrices that intersect the blocker once each. Consider such a permutation matrix using the $f$ from the first row and the $e$ from the last column below.

$$\left[\begin{array}{c|c|c|c|c|c|c|c|c|c}
a&b&c&d&\cellcolor{red!20}e&\cellcolor{yellow!20}f&\cellcolor{red!20}g&h&i&j\\ \hline
b&c&d&\cellcolor{green!20}e&\cellcolor{red!20}f&\cellcolor{red!20}g&\cellcolor{red!20}h&i&j&a\\ \hline
c&d&\cellcolor{green!20}e&f&\cellcolor{red!20}g&\cellcolor{red!20}h&\cellcolor{red!20}i&j&a&b\\ \hline
d&\cellcolor{green!20}e&f&g&\cellcolor{red!20}h&\cellcolor{red!20}i&\cellcolor{red!20}j&a&b&c\\ \hline
\cellcolor{green!20}e&f&g&h&i&j&\cellcolor{red!20}a&b&c&d\\ \hline
f&g&h&i&j&a&\cellcolor{red!20}b&c&d&\cellcolor{green!20}e\\ \hline
g&h&i&j&a&b&\cellcolor{red!20}c&d&\cellcolor{green!20}e&f\\ \hline
h&i&j&a&b&c&\cellcolor{red!20}d&\cellcolor{green!20}e&f&g\\ \hline
i&j&a&b&c&d&\cellcolor{green!20}e&f&g&h\\ \hline
j&a&b&c&\cellcolor{green!20}d&e&f&g&h&i
\end{array}\right]$$

Using the $m+1$ through $n-2$ positions from row one, we can obtain $k-2=n-m-2$ additional linearly independent $123$-avoiding permutation matrices that intersect the blocker once and that use a unique position. Consider the below example.

$$\left[\begin{array}{c|c|c|c|c|c|c|c|c|c}
a&b&c&d&\cellcolor{red!20}e&\cellcolor{red!20}f&\cellcolor{red!20}g&\cellcolor{green!20}h&i&j\\ \hline
b&c&d&e&\cellcolor{red!20}f&\cellcolor{red!20}g&\cellcolor{red!20}h&i&j&\cellcolor{green!20}a\\ \hline
c&d&e&f&\cellcolor{red!20}g&\cellcolor{red!20}h&\cellcolor{red!20}i&j&\cellcolor{green!20}a&b\\ \hline
d&e&f&g&\cellcolor{red!20}h&\cellcolor{red!20}i&\cellcolor{yellow!20}j&a&b&c\\ \hline
e&f&g&h&i&\cellcolor{green!20}j&\cellcolor{red!20}a&b&c&d\\ \hline
f&g&h&i&\cellcolor{green!20}j&a&\cellcolor{red!20}b&c&d&e\\ \hline
g&h&i&\cellcolor{green!20}j&a&b&\cellcolor{red!20}c&d&e&f\\ \hline
h&i&\cellcolor{green!20}j&a&b&c&\cellcolor{red!20}d&e&f&g\\ \hline
i&\cellcolor{green!20}j&a&b&c&d&e&f&g&h\\ \hline
\cellcolor{green!20}j&a&b&c&d&e&f&g&h&i
\end{array}\right]$$

Thus far, we have used $3$ positions from column $n$ in our constructions. There are $n-t-3$ additional positions in this column that we may utilize, each of which will result in a permutation matrix that contains a unique position of the matrix. Consider two such examples below.

$$\left[\begin{array}{c|c|c|c|c|c|c|c|c|c}
a&b&c&d&\cellcolor{red!20}e&\cellcolor{red!20}f&\cellcolor{red!20}g&h&\cellcolor{green!20}i&j\\ \hline
b&c&d&e&\cellcolor{red!20}f&\cellcolor{red!20}g&\cellcolor{red!20}h&\cellcolor{green!20}i&j&a\\ \hline
c&d&e&f&\cellcolor{yellow!20}g&\cellcolor{red!20}h&\cellcolor{red!20}i&j&a&b\\ \hline
d&e&f&\cellcolor{green!20}g&\cellcolor{red!20}h&\cellcolor{red!20}i&\cellcolor{red!20}j&a&b&c\\ \hline
e&f&\cellcolor{green!20}g&h&i&j&\cellcolor{red!20}a&b&c&d\\ \hline
f&\cellcolor{green!20}g&h&i&j&a&\cellcolor{red!20}b&c&d&e\\ \hline
\cellcolor{green!20}g&h&i&j&a&b&\cellcolor{red!20}c&d&e&f\\ \hline
h&i&j&a&b&c&\cellcolor{red!20}d&e&f&\cellcolor{green!20}g\\ \hline
i&j&a&b&c&d&\cellcolor{green!20}e&f&g&h\\ \hline
j&a&b&c&d&\cellcolor{green!20}e&f&g&h&i
\end{array}\right]
\text{ \ \ and \ \ }
\left[\begin{array}{c|c|c|c|c|c|c|c|c|c}
a&b&c&d&\cellcolor{red!20}e&\cellcolor{red!20}f&\cellcolor{red!20}g&h&\cellcolor{green!20}i&j\\ \hline
b&c&d&e&\cellcolor{red!20}f&\cellcolor{red!20}g&\cellcolor{red!20}h&\cellcolor{green!20}i&j&a\\ \hline
c&d&e&f&\cellcolor{red!20}g&\cellcolor{red!20}h&\cellcolor{yellow!20}i&j&a&b\\ \hline
d&e&f&g&\cellcolor{red!20}h&\cellcolor{red!20}i&\cellcolor{red!20}j&a&b&\cellcolor{green!20}c\\ \hline
e&f&g&h&i&\cellcolor{green!20}j&\cellcolor{red!20}a&b&c&d\\ \hline
f&g&h&i&\cellcolor{green!20}j&a&\cellcolor{red!20}b&c&d&e\\ \hline
g&h&i&\cellcolor{green!20}j&a&b&\cellcolor{red!20}c&d&e&f\\ \hline
h&i&\cellcolor{green!20}j&a&b&c&\cellcolor{red!20}d&e&f&g\\ \hline
i&\cellcolor{green!20}j&a&b&c&d&e&f&g&h\\ \hline
\cellcolor{green!20}j&a&b&c&d&e&f&g&h&i
\end{array}\right]$$

After adding and rewriting, in total we have constructed $(n-1)^2+1-(t+2)(k+1)$ linearly independent $123$-avoiding permutation matrices that intersect the blocker exactly once, as desired. \end{proof}

Our results relating to the upper and lower bound for the dimension of a face of $\Omega_n(\overline{123})$ provide insight into the geometric properties of the polytope, though there is still work to be done in terms of more precisely determining the dimension of the faces of $\Omega_n(\overline{123})$.


\begin{thebibliography}{1}
\bibitem{BC1} Brualdi, R.A., \& Cao, L. (2021). Pattern-avoiding $(0,1)$-matrices and bases of permutation matrices. \textit{Discrete Appl. Math.}, 304, 196-211.
\bibitem{BC2} Brualdi, R.A., \& Cao, L. (2022). Blockers of pattern avoiding permutation matrices. \textit{Australasian J. Combin.}, 83(2), 274-303.
\bibitem{BC} Brualdi, R.A., \& Cao, L. (2023).  $123$-Forcing matrices. \textit{Australasian J. Combin.}, 86(1), 169-186.
\bibitem{BC3} Brualdi, R.A., \& Cao, L. (2023).  $123$-Avoiding doubly stochastic matrices. \textit{Linear Algebra Appl.}, 1-33.
\bibitem{FK1933} K\"{o}nig, D. (1933). \"{U}ber trennende knotenpunkte in graphen (nebst anwendundungen auf determinanten und matrizen). \textit{Sectio Scientiarum Mathematicarum (Szeged)}, 6, 155-179.


\end{thebibliography}
\end{document}